\documentclass[12pt]{article}

\usepackage{amssymb,amsmath,amsthm}
\usepackage{verbatim}
\usepackage{fullpage}
\usepackage{calrsfs}
\usepackage{color}
\definecolor{aleacolour}{rgb}{0.09,0.32,0.44} 
\usepackage[breaklinks=true,pdfborder={0 0 0},colorlinks=true,pdfborderstyle={},linkcolor=aleacolour,citecolor=aleacolour,urlcolor=blue]{hyperref}

\newtheorem*{theorem*}{Theorem}
\newtheorem{lemma}{Lemma}

\theoremstyle{remark}
\newtheorem*{corollary*}{Corollary}
\newtheorem*{problem*}{Problem}

\newcommand{\bigA}{\mathcal{A}}

\newcommand{\N}{\mathbb{N}}
\newcommand{\PP}{\mathbb{P}}
\newcommand{\E}{\mathbb{E}}

\DeclareMathOperator{\Med}{Med}
\DeclareMathOperator{\cov}{cov}

\title{Counting bats}
\author{Itai Benjamini \and Gady Kozma}

\date{}

\begin{document}

\maketitle

 


Assume $G$ is an infinite graph (``the cave'') which is recurrent for
the simple random walk (SRW). Several independent walkers (``the bats'') are
performing SRW on $G$ simultaneously with the same clock with starting
vertex $o$. $G$ is not known to you (hence the cave metaphor), it is too dark to see $G$). The only information given to you
is the set of return times to $o$, though you do not know how many
walkers returned at any given time, only if this number is 0 or
positive. Can you almost surely  tell how many walkers
are there, by only observing the times $o$ is occupied?
 
\begin{theorem*}
Almost surely it is possible to tell how many walkers are there, by observing the times $o$ is occupied.
\end{theorem*} 

Formally, there is a function $\bigA:\{0,1\}^\N\to\N$ which,
given the visits of the walkers outputs their number, and is correct
with probability 1. Again, $\bigA$ does not depend on the graph. For more bat-related results, see our paper \cite{BKLRT}.

\begin{corollary*}
There is no pair of recurrent  infinite graphs
so that the return times of two independent SRW's on one of the graphs are absolutely continuous
to the return times of one SRW on the other graph.
\end{corollary*}

\begin{problem*}The algorithm in the proof seems far from efficient.
Give lower and upper bounds and suggest improved or optimal algorithms.
\end{problem*}

(``the algorithm'' here is simply the function $\bigA$, which is not really an algorithm in the computer sense: it does not ``run'' or ``stop''. Nevertheless one can find reasonable algorithmic versions of the problem and investigate them)
 
\begin{problem*}We do not know if reversibility is important (it is definitely used in our proof). So we ask: is there an algorithm that gives the right number of walkers for any recurrent Markov chain, without knowledge of the Markov chain?
\end{problem*}

\subsection*{Proof}
The first step is to reconstruct the
distribution of returns of a \emph{single} walker, no matter how many
walkers one actually examines.
\begin{lemma}\label{lem:one}There is an algorithm to reconstruct
\[
p(n)=\PP(\textrm{a single walker returns to $o$ at time $n$})
\]
with no knowledge of the graph structure.
\end{lemma}
\begin{proof}
We fix some large $T_1$ and wait till you see a time interval of length
$\ge T_1$ with no visits to $o$ and look at the first return after this
long returns-free interval, denoted by $s_1$. Let $E_1$ be the event
that there is a return at $s_1+n$. Continue similarly: choose some
large $T_2$, let $s_2$ be the first time after $s_1+n$ when a
returns-free intervals longer than $T_2$ finished, and let $E_2$ be
the event that there is a return at $s_2+n$. Etc. For concreteness,
fix $T_i=2^i$.

Write now $E_i=G_i\cup B_i$ where $G_i$ (the ``good'' event) is the
event that the walker which returned at time $s_i$ also returned at
$s_i+n$, and $B_i$ (the ``bad'' event) is the event that another
walker returned at time $s_i+n$. We would have liked to sample $G_i$,
which are exactly i.i.d.\ variables with probability $p(n)$, but we
can only sample $E_i$. So we need to show that $B_i$ are rare.

The key observation follows from a quantitative non
concentration of return times established in  \cite{GN}: on any graph,
the probability that a walker returned for the first time at time $t$,
conditioned on not having returned before $t$,
is $\le (C\log t)/t$. Denote therefore $B_i(j)$ the event
that it is the $j^{\textrm{th}}$ walker that returned at time $s_i+n$
($B$ is a ``bad'' event, so we assume that the $j^\textrm{th}$ walker
did not return at time $s_n$ or that there were two returns at $s_n$).
Let $r$ be the last visit of the $j^{\textrm{th}}$ walker to $o$
before $s_i$. By definition, this means that $r<s_i-2^i$. We can now write
\begin{align*}
\PP(B_i(j)\,|\,s_i,r)&=\PP(\textrm{a walker returned at time
  $s_i-r+n$}\\
&\qquad\qquad|\,
  \textrm{not returning in the first $s_i-r$ steps})\\
&\le\sum_{j=0}^n\PP(\textrm{a walker returned at time
    $s_i-r+j$}\\
&\qquad\qquad|\,\textrm{not returning in the first $s_i-r+j-1$
    steps})\\
\textrm{by \cite{GN}}\qquad
&\le\sum_{j=1}^nC\frac{\log(s_i-r+j)}{s_i-r+j}\le \frac{Cni}{2^i}
\end{align*}
Integrating over $r$ and $s_i$ gives that $\PP(B_i(j))\le Cni2^{-i}$,
and summing over $j$ (which has $k$ possibilities, where $k$ is the
(unknown) number of walkers) gives $\PP(B_i)\le Ckni2^{-i}$. We see
that these numbers are summable, so only a finite number of $B_i$
occur. This means that $p(n)$ may be calculated by 
\[
p(n)=\lim_{\ell\to\infty}\frac{|\{i\le\ell:E_i\}|}{\ell}
\]
which is the algorithm sought for.
\end{proof}
With the distribution of returns estimated, we now have a relatively
easy task: we have a known variable, the number of returns of a single
walker. We are given a sample of a union of $k$ independent copies of it
and we need to estimate $k$. Taking the number of actual returns up to
time $t$ and dividing by the (known) expectation for a single walker,
would give a variable with expectation $k$. It would be natural to
assume that if we repeat this experiment with times $t_i$ growing
sufficiently fast, the resulting variables would be approximately
independent and hence it would be possible to calculate $k$ by the
limit of the running average. The only difficulty is to explain what
does ``sufficiently fast'' means, and it turns out that this must
depend on the graph. The following lemma essentially claims that this
scheme works if one takes $t_i$ to be the median of the
$(2^i)^{\textrm{th}}$ return of a single walker to $o$.

\begin{lemma}\label{lem:median}
Let $X_{n}$ be i.i.d.\ $\mathbb{N}$-valued random variables, and
let $S_{n}$ be the corresponding random walk
\[
S_{n}=X_{1}+\dotsc+X_{n}.
\]
Let $M_n$ be the median of
$S_n$ i.e.
\[
M_{n}=\Med(S_{n})\Longrightarrow n=\max\{i:\mathbb{P}(S_{i}\le M_{n})>\tfrac{1}{2}\}.
\]
Define 
\[
Y_{n}=\frac{1}{n}\max\{i:S_{i}\le M_{n}\}
\]
Then
\[
\frac 1N\sum_{n=1}^N\frac{Y_{2^n}}{\E Y_{2^n}}\to 1
\]
as $N\to\infty$.\end{lemma}
\begin{proof}
By definition, $\Med Y_{n}=1$. It is easy to conclude from that that
$Y_{n}$ form a precompact (tight) family of variables. Indeed,
\[
Y_{n}>\lambda\iff\max\{i:S_{i}\le M_{n}\}>n\lambda\iff S_{\left\lfloor n\lambda\right\rfloor +1}\le M_{n}\;.
\]
Since $S$ is an increasing random walk, to be smaller than $M_n$ at time $n\lambda$ its increments must be smaller than $M_n$ on any block of variables between $0$ and $n\lambda$, and disjoint blocks are independent. So we get
\begin{equation}
\mathbb{P}(Y_{n}>\lambda)=
\mathbb{P}(S_{\lfloor n\lambda\rfloor+1}\le M_n)\le
\mathbb{P}(X_{(k-1)n+1}+\dotsb+X_{kn}\le M_{n}\forall k\le\lfloor\lambda\rfloor)\le
2^{-\left\lfloor \lambda\right\rfloor }.\label{eq:Ydecaysepx}
\end{equation}
We will use the second moment method so we need to estimate 
\[
\mathbb{E}(Y_{n}Y_{m})-\mathbb{E}(Y_{n})\mathbb{E}(Y_{m})
\]
say, for $m<n$ (both will be powers of two but let us not record
this fact in the notation). Let therefore $\lambda=\left\lfloor (n/m)^{1/3}\right\rfloor $.
Define the event $\mathcal{B}$ (the ``bad'' event) to be the event
that one of the following happened:
\begin{enumerate}
\item $Y_{m}\ge\lambda$.
\item $S_{mY_{m}+1}>M_{m\lambda^{2}}$ (note the $+1$ in the index ---
we are taking here the first time $S_{i}$ raises above $M_{m}$).
\end{enumerate}
The probability of the first clause is estimated by
(\ref{eq:Ydecaysepx}) to be $\le 2^{-\lambda}$,
so let us estimate the probability of the second minus the first, i.e.\ the probability that $Y_m<\lambda$ but $S_{mY_m+1}>M_{m\lambda^2}$. Because $S$ is increasing, if $mY_m<m\lambda$ then $mY_m+1\le m\lambda$ and then $S_{mY_m+1}\le S_{m\lambda}$ so we can write
\[
\mathbb{P}(Y_{m}<\lambda,S_{mY_{m}+1}>M_{m\lambda^{2}})\le\PP(S_{m\lambda}>M_{m\lambda^2})=:p.
\]
But then we can write
\begin{align*}
\frac{1}{2} & <\mathbb{P}(S_{m\lambda^{2}}\le M_{m\lambda^{2}})\\
 & \le\mathbb{P}\left(X_{(k-1)m\lambda+1}+\dotsb+X_{km\lambda}\le M_{m\lambda^{2}}\quad\forall k\le\lambda\right)\\
 & \le(1-p)^{\lambda}
\end{align*}
so $p\le C/\lambda$. Totally we get
\[
\mathbb{P}(\mathcal{B})\le C/\lambda.
\]
With the estimate (\ref{eq:Ydecaysepx}) this gives
\begin{equation}
\mathbb{E}(Y_{m}\mathbf{1}_{\mathcal{B}}) \le
\sum_{k=1}^{\infty}\mathbb{P}(\mathcal{B}\cap\{k-1<Y_{m}\le k\})\cdot k\le
\sum_{k=1}^\infty Ck\min\Big\{\frac1\lambda,2^{-k}\Big\}\le
\frac{C(\log\lambda)^2}{\lambda}.\label{eq:bad_who_cares}
\end{equation}
We need a similar estimate for $Y_nY_m\mathbf{1}_{\mathcal{B}}$ and for this we need to estimate $\mathbb{E}(Y_n\,|\,\mathcal{B}\cap\{k-1<Y_m\le k\}$. We write $nY_n=mY_m+1+Z$ and note that $Z$ is the number of steps our random walk needed to get from $S_{mY_m+1}$ to $n$ so it is stochastically dominated by $nY_n$, even after conditioning over $\mathcal{B}\cap\{k-1<Y_m\le k\}$ (which is an event that looks at the random walk only up to $mY_m+1$). So
\[
\mathbb{E}(Y_n|\mathcal{B}\cap\{Y_m=y\})\le y\frac{m}{n}+C
\]
which we use to show
\begin{align}
\mathbb{E}(Y_{n}Y_{m}\mathbf{1}_{\mathcal{B}}) & \le
\sum_{k=1}^{\infty}\mathbb{P}(\mathcal{B}\cap\{k-1<Y_{m}\le k\})\cdot k\cdot(k(m/n)+C)\nonumber \\
 & \le \sum_{k=1}^\infty Ck^2 \min\Big\{\frac1\lambda,2^{-k}\Big\}
\le\frac{C(\log\lambda)^3}{\lambda}\;.\label{eq:bad_YmYn}
\end{align}
This finishes our treatment of the event $\mathcal{B}$.

We now restrict our attention to $\neg\mathcal{B}$. Let therefore
$\omega$ be some atom of the $\sigma$-field spanned by $Y_{m},X_{1},\dotsc,X_{mY_{m}+1}$
such that $\omega\not\subset\mathcal{B}$, and write
\begin{equation}
\mathbb{E}(Y_{n}|\omega)=Y_{m}\cdot\frac{m}{n}\\
+\frac{1}{n}\mathbb{E}(\max\{i:X_{mY_{m}+2}+\dotsb+X_{mY_{m}+1+i}\le M_{n}-(S_{mY_{m}+1})\}\,|\,\omega)\label{eq:Yn|ome}
\end{equation}
The first term we bound by $C/\lambda^2$ (because of the first clause
in the definition of $\mathcal{B}$). For the second, we note that
$X_{2^{m}Y_{m}+2},\dotsc$ has the same distribution as $X_{1},\dotsc$
(again, conditioning over $\omega$ does not change this fact) so this
term is bounded above by $\mathbb{E}Y_{n}$ and bounded below by 
\[
\frac{1}{n}\mathbb{E}(\max\{i:S_{i}\le M_{n}-M_{m\lambda^{2}}\})
\]
by the second clause in the definition of $\mathcal{B}$. Hence we
need to estimate the variable
\[
\max\{i:S_{i}\le M_{n}\}-\max\{i:S_{i}\le M_{n}-M_{m\lambda^{2}}\}=|\{i:M_{n}-M_{m\lambda^{2}}<S_{i}\le M_{n}\}|
\]
But this variable is stochastically dominated simply by $m\lambda^{2}Y_{m\lambda^{2}}$
because it is the number of steps our random walk needs to traverse an interval $\le M_{m\lambda^{2}}$.
Combining both parts of (\ref{eq:Yn|ome}) gives 
\[
\mathbb{E}(Y_{n})-C/\lambda\le\mathbb{E}(Y_{n}|\omega)\le C/\lambda+\mathbb{E}(Y_{n})
\]
which we multiply by $Y_{m}$ and integrate over $\neg\mathcal{B}$
to get
\[
|\mathbb{E}(Y_{m}Y_{n}\mathbf{1}_{\neg\mathcal{B}})-\mathbb{E}(Y_{n})\mathbb{E}(Y_{m}\mathbf{1}_{\neg\mathcal{B}})|\le\frac{C}{\lambda}\mathbb{E}(Y_{m}\mathbf{1}_{\neg\mathcal{B}})\le\frac{C}{\lambda}.
\]
With (\ref{eq:bad_who_cares}), (\ref{eq:bad_YmYn}) we get
\begin{equation}
|\mathbb{E}(Y_{m}Y_{n})-\mathbb{E}(Y_{m})\mathbb{E}(Y_{n})|\le\frac{C(\log\lambda)^3}{\lambda}.\label{eq:covariance}
\end{equation}
This finishes the lemma: define 
\[
A_{N}=\sum_{i=1}^{N}Y_{2^{i}}
\]
and estimate $\mathbb{V}A_{n}$. We get
\begin{align*}
\mathbb{V}A_{N} & =
\sum_{i=1}^{N}\mathbb{V}Y_{2^{i}}+2\sum_{1\le i<j\le N}\cov(Y_{2^{i}},Y_{2^{j}})\\
 & \le\sum_{i=1}^{N}C+2\sum_{1\le i<j\le N}C\cdot2^{(i-j)/3}|i-j|^3\le CN
\end{align*}
where the bound for $\mathbb{V}Y_{2^{i}}$ comes from the exponential
decay (\ref{eq:Ydecaysepx}), and the bound for the covariances is
exactly (\ref{eq:covariance}), recall that $\lambda$ was defined
by $\left\lfloor (n/m)^{1/3}\right\rfloor $. On the other hand, $\Med Y_{2^{i}}=1$
so $\mathbb{E}Y_{2^{i}}\ge\frac{1}{2}$ and $\mathbb{E}A_{N}\ge\frac{1}{2}N$.
This gives that $A_{N}/\mathbb{E}A_{N}$ is concentrated. Using Markov's
inequality gives
\[
\mathbb{P}\left(\left|\frac{A_{N}}{\mathbb{E}A_{N}}-1\right|>\epsilon\right)\le
\mathbb{P}(|A_{N}-\mathbb{E}A_{N}|>c\epsilon N)\le
\frac{C}{\epsilon^2 N}.
\]
This means that these events happen only finitely many times on any
reasonable subsequence (e.g.\ $N^{2}$) and due to $\mathbb{E}Y_{2^{i}}\approx1$
and monotonicity of $A_{N}$ the convergence may be extended from a
subsequence to all $N$.\end{proof}

We are almost done, we just need to handle double returns to $o$,
for which we have the following simple lemma.
\begin{lemma}\label{lem:picoez}
Let $X_1$ and $X_2$ be two independent walkers on an infinite graph $G$, and
let $t>0$. Then
\[
\E(|\{s\le t:X_1(s)=X_2(s)=o\}|)\le C\Big(\E(|\{s\le
t:X_1(s)=o\}|)\Big)^{2/3}.
\]
\end{lemma}
\begin{proof}
Denote $M=\E(|\{s\le t:X_1(s)=o\}|)$. On any
infinite graph, $\PP(X_2(s)=o)\le C/\sqrt{s}$ (see e.g.\ \cite{C99}). Hence we write
\begin{align*}
\E(|\{M^{2/3}&\le s\le t:X_1(s)=X_2(s)=o\}|)=\sum_{s=M^{2/3}}^t
\PP(X_1(s)=o)^2\\
&\le \sum_{s=M^{2/3}}^t\PP(X_1(s)=o)\cdot\frac{C}{\sqrt{s}}\le 
CM^{-1/3}\sum_{s=1}^t \mathbb{P}(X_1(s)=o)=
CM^{2/3}.
\end{align*}
Since the number of visits up to time $M^{2/3}$ is definitely bounded
by $M^{2/3}$, we are done.
\end{proof}
The theorem now follows easily. By lemma \ref{lem:one} we may
calculate the median $M_n$ of the $n^{\textrm{th}}$ return of a single
walker to $o$. Defining
\[
Y_n^i=\frac 1n |\{\textrm{visits of walker $i$ until $M_n$}\}|\qquad Y_n=\sum_i Y_n^i
\]
we can use lemma \ref{lem:median} (the random walk $S$ in lemma \ref{lem:median} is defined by $S_n$ being the time of the $n^\textrm{th}$ visit of the walker to $o$ and then the $M_n$ of lemma \ref{lem:median} are the same as here, and the $Y_n$ of lemma \ref{lem:median} are the $Y_n^i$ here). We get
\[
\frac 1N\sum_{n=1}^N\frac{Y_{2^n}}{\E Y_{2^n}^i}\to k\;.
\]
We cannot measure $Y_n$ directly, since if two
walkers returned to $o$ at the same time, they contribute $2$ to the sum
but we cannot see that. Nevertheless, if we define 
\[
\widetilde{Y}_n=\frac 1n |\{1\le t \le M_n: \exists j, X_j(t)=o\}|
\]
then $\widetilde{Y}_n$ can be measured, and
\[
|Y_n-\widetilde{Y}_n|\le \frac 1n \sum_{i,j}|\{1\le t\le M_n:
X_i(t)=X_j(t)=o\}|
\]
and each term is bounded by lemma \ref{lem:picoez} by $(\E (nY_n))^{2/3}
/n$. Since $\E Y_n\le C$ this gives that 
\[
|Y_n-\widetilde{Y}_n|\le Ck^2n^{-1/3}.
\]
We get
\[
\frac 1N\sum_{n=1}^N\frac{\widetilde{Y}_{2^n}}{\E Y_{2^n}^i}\to k
\]
and the theorem is proved.\qed
%



 \subsection*{Acknowledgements}
Both authors supported by their respective Israel Science Foundation grants.


\bigskip
\begin{flushright}
\footnotesize\obeylines
   \textsc{Weizmann Institute}
  \textsc{Rehovot, Israel}
  \textsc{E-mail:} \texttt{itai.benjamini@weizmann.ac.il ; gady.kozma@weizmann.ac.il}
\end{flushright}

\end{document}